\documentclass[10pt]{article}
\usepackage{amsmath,amsthm,amsfonts,amssymb,amscd, amsxtra}
\usepackage[margin=2.5cm,nohead]{geometry}
\usepackage{multicol}
\usepackage{amsfonts}
\usepackage{textcomp}

\theoremstyle{plain}
\newtheorem{theorem}{Theorem}[section]
\newtheorem{lemma}{Lemma}[section]
\newtheorem{definition}{Definition}[section]

\newtheorem{proposition}{Proposition}[section]

\newtheorem{remark}{Remark}[section]

\newcommand{\beq}{\begin{equation}}
\newcommand{\eeq}{\end{equation}}
\newcommand{\beqa}{\begin{eqnarray}}
\newcommand{\eeqa}{\end{eqnarray}}
\newcommand{\beqas}{\begin{eqnarray*}}
\newcommand{\eeqas}{\end{eqnarray*}}

\makeatletter
\renewcommand*{\@biblabel}[1]{\hfill#1.}
\makeatother

\begin{document}

\title{An interior subgradient and a proximal linearized method for DC programming}

\author{
Valdin\^es Leite de Sousa J\'unior\thanks{Universidade Federal do Cariri, Juazeiro do Norte1, CE, BR  (Email: {\tt valdines.leite@ufca.edu.br}).}
}
\date{}
\maketitle
\centerline{\today}
\vspace{.2cm}

\noindent

\begin{abstract}
In this paper, we study the convergence of an interior subgradient and proximal methods for a $DC$ (difference of convex functions) constrained minimization problem.
\end{abstract}

\section{Introduction}\label{intro}
In the past decades, the interest on $DC$ programming has been increasing considerably. Since then, it has become a very promising research field with several developments in many research lines, theoretical and application, see for instance \cite{Hiriart1989,Hiriart1985,Hiriart1988,Elhilali1996}. Recently, some authors have been proposed some algorithms and numerical experiments do study $DC$ optimization problems in a lot of settings, even in Riemann Manifolds, see   \cite{CruzNeto2018,JCO2016,Sun2003,Moudafi2006,JC2015,JCX2018,JCX2020,JCXY2020}.

The problem of finding the critical points of a function is a very common problem in Optimization. In general, the desirable result is to find a zero of the subdifferential of that function. The main goal of this paper is to study this kind of problem for a special class of a nonconvex function, namely $DC$ (difference of convex functions) functions. To do so, we will analyze the convergence of two different algorithms. The first one was based on the interior gradient methods presented by \cite{Auslander2006}, but in our case, the method is applied for a $DC$ function instead of a convex function. The second one was considered in \cite{CruzNeto2018} with a particular choice of the proximal distance as second-order homogeneous proximal distances and Bregman distances. In our case, we considered the same method as in \cite{CruzNeto2018} but we are going to use a different approach with respect to the objective function and the choice of a kind of type proximal distance. In both cases, we prove that every accumulation point of its generated sequences, if any, is a critical point of a $DC$ function over a nonempty, closed and convex set. Furthermore, with some additional assumptions, the whole sequence converges to a critical point of a $DC$ function. 

The organization of the paper is as follows. In Section~\ref{sec1}, some notations and basic results used throughout the paper are presented. In Section~\ref{sec12}, the two algorithms studied in this paper are presented and the main results are stated and proved. Some final remarks are made in Section~\ref{sec:conclusion}.

\section{Preliminary}\label{sec1}
 
 In this section, we present several concepts of non-smooth analysis that will be useful throughout this presentation.
 
 The subdifferential of a convex  lower semicontinuous function $f:\mathbb{R}^n\to\mathbb{R}$ at $x$, is defined
by $$\partial f(x)=\{v\in\mathbb{R}^n:\forall y\in\mathbb{R}^n, f(y)\geq f(x)+\langle v,y-x\rangle\}.$$ 
If $f$ is strongly convex with modulus $\rho>0$, it is well known that, for all $v\in\partial f(x)$,
$$
f(y)\geq f(x)+\langle v,y-x\rangle+\frac{\rho}{2}\|y-x\|^2,\quad \forall y\in\mathbb{R}^n,
$$ 
and its subdifferential $\partial f$ is strongly monotone with modulus $\rho$, i. e., for any $x,y\in\mathbb{R}^n, v\in\partial f(y)$ and $u\in\partial f(x)$, we have
$$
\langle v-u,y-x\rangle\geq\rho\|y-x\|^2.
$$
 
Let $f:\mathbb{R}^n\rightarrow\mathbb{R}$ be a locally Lipschitz function at $\bar x\in \mathbb{R}^n$ with constant $L>0$ and $v\in\mathbb{R}^n$. The {\it Clarke's directional derivative}
\cite[page 25]{Clarke1990} of $f$ at $\bar x$ in the direction $v$, denoted by $f^{\circ}(\bar x;v)$, is defined as 
$$
f^{\circ}(\bar x;v):=\displaystyle\limsup_{t\downarrow0~y \rightarrow \bar x}\frac{f(y+tv)-f(y)}{t}, 
$$
and {\it Clarke's subdifferential} \cite[page 27]{Clarke1990} of $f$ at $\bar x$, denoted by $\partial^{\circ}f(\bar x)$, is defined as
$$\partial^{\circ}f(\bar x):=\left\{w\in\mathbb{R}^n~:~f^{\circ}(\bar x;v)\geq\langle w,v\rangle,~ ~v\in\mathbb{R}^n\right\}.$$ If $f$ is convex, the Clarke's subdifferential coincides with the classical subdifferential $\partial f$.
 
 A lower semicontinuous function $f:\mathbb{R}^n\to\mathbb{R}$, is called a $DC$ \textit{function} when there exist convex functions $g$ and $h$ such that,
\begin{equation}\label{DCfunction}
f(x)=g(x)-h(x), \forall x\in\mathbb{R}^n.
\end{equation}
The functions $g$ and $h$ are commonly called \textit{components functions of} $f$. It is well known that a necessary condition for $x\in\mathbb{R}^n$ to be a local
minimizer of a $DC$ function $f$ is $\partial h(x)\subset\partial g(x)$. In general, this condition is hard to be reached, often such condition is replaced by a relaxed one, namely points that satisfies $\partial g(x)\cap\partial h(x)\neq\emptyset$. Inspired by this condition and for the other definitions of critical points for constrained problems, we have the following definition. 

\begin{definition}\label{defi1}
Let $D$ be a closed and convex set and  $f:\mathbb{R}^n\to\mathbb{R}$ be a $DC$ function as in \eqref{DCfunction}. We say that a point $x^*\in D$ is a \textit{critical} point of $f$ in $D$
if, there exist $v\in\partial g(x^*)$ and $u\in\partial h(x^*)$ such that 
$$\langle v-u, y-x^*\rangle\geq0,\quad \forall y\in D.$$ We denote by $\mathcal{S}_{D}^*(f)$, the set of the critical points of $f$ in $D$.
\end{definition}

In terms of Clarke's directional derivative, the previous definition can be interpreted as follows.

\begin{definition}\label{defi2}
Let $D$ be a closed and convex set and  $f:\mathbb{R}^n\to\mathbb{R}$ be a $DC$ function as in \eqref{DCfunction}. We say that a point $x^*\in D$ is a \textit{Clarke-critical} point of $f$ in $D$
if,  
$$f^{\circ}(x^*; y-x^*)\geq0,\quad \forall y\in D.$$ We denote by $\mathcal{S}_{D}^{\circ}(f)$, the set of the Clarke-critical points of $f$ in $D$.
\end{definition}

\begin{remark} Note that if $D=\mathbb{R}^n$ in the Definition \ref{defi1}, $\partial g(x^*)\cap\partial h(x^*)\neq\emptyset$. Furthermore, if $h$ is continuously differentiable, $\mathcal{S}_{D}^*(f)\subset\mathcal{S}_{D}^{\circ}(f)$, thanks to \cite[Corollary 1, page 39]{Clarke1990}.

\end{remark}
 
 In our approach, we choose a proximal distance $d : \mathbb{R}^n\times\mathbb{R}^n\to\mathbb{R}_+\cup\{+\infty\}$ as the regularization term. Such a well-known distance 
allows us to analyze the convergence of the algorithm under various settings. Following \cite{Auslander2006}, let us recall the definition of the proximal and induced proximal 
distances.
\begin{definition}\label{proximaldistance}
	A function $d : \mathbb{R}^n\times\mathbb{R}^n\to\mathbb{R}_+\cup\{+\infty\}$ is called a proximal
	distance with respect to an open nonempty convex set $C\subset\mathbb{R}^n$ if for each $y\in C$ it
	satisfies the following properties:
	\begin{description}
		\item [(\textbf{d1})] $d(\cdot, y)$ is proper, lsc, convex, and $C^1$ on $C$;
		\item [(\textbf{d2})] $\emph{dom}~d(\cdot, y)\subset\bar{C}$ and $\emph{dom}~\partial_1d(\cdot, y)={C}$, where $\partial_1d(\cdot, y)$ denotes the subgradient
		map of the function $d(\cdot, y)$ with respect to the first variable;
		\item [(\textbf{d3})] $d(\cdot, y)$ is level bounded on $\mathbb{R}^n$, i.e., $\lim_{\|u\|\to+\infty} d(u, y)=+\infty$;
		\item [(\textbf{d4})] $d(y, y) = 0$.
	\end{description}
\end{definition}

For each $y\in C$, let $\nabla_1d(\cdot,y)$ denote the gradient map of the function $d(\cdot, y)$ with respect to the first variable. Note that by definition $d(\cdot,\cdot)\geq0$, and from \textbf{(d3)} the global minimum of $d(\cdot,y))$ is obtained at $y$, which shows that $\nabla_1d(y,y)=0$. We denote by $\mathcal{D}(C)$ the family of functions $d$ satisfying \textbf{(d1)-(d4)}.
	
	Next, following the approach presented in~\cite{Auslander2006}, we associate to a given $d\in \mathcal{D}(C)$ a corresponding induced distance $H$ that satisfies some desirable properties.

\begin{definition}\label{proximaldistance21}
	Given $C\subset\mathbb{R}^n$, open and convex, and $d\in \mathcal{D}(C)$, a function $H:\mathbb{R}^n\times\mathbb{R}^n\to\mathbb{R}_+\cup\{+\infty\}$ is called the induced proximal distance to $d$ if $H$ is finite
	valued on $C\times C$ and for each $x, y\in C$ satisfies the following properties:
	\begin{description}
		\item[(\textbf{H1})] $H(x,x)=0$;
		\item[(\textbf{H2})] $\left\langle z-y,\nabla_1 d(y,x)\right\rangle\leq H(z,x)-H(z,y)$, $ z\in C$.
	\end{description}
\end{definition}

We write $(d,H)\in \Phi(C)$ to quantify the triple $[C, d,H]$ that satisfies the premises of Definition~\ref{proximaldistance21}. Similarly, we write $(d,H)\in \Phi(\bar{C})$ 
for the triple $[\bar{C}, d,H]$ whenever there exists $H$, which is finite valued on $\bar{C}\times C$, satisfies \textbf{(H1)}-\textbf{(H2)} for any $z\in C$, and is such that
$ z\in \bar{C}$ has $H(z,\cdot)$ level bounded on $C$. Clearly, one has $\Phi(\bar{C})\subset \Phi(C)$.
For examples and a thorough discussion about proximal and induced proximal distances see, for instance, 
\cite{Auslander2006,BurachikSIAM2010}. 

Before we introduce the main results of the present paper, we recall the following well-known results of nonnegative sequences.
\begin{lemma}[see \cite{Polyak1987}]\label{lemmaone} Let $\{u^k\}$, $\{\alpha_k\}$, and $\{\beta_k\}$ be nonnegative sequences of real numbers satisfying $u^{k+1}\leq(1+\alpha_k)u^k+\beta_k$ such that
	$\sum_k\alpha_k<\infty$ and $\sum_k\beta_k<\infty$. Then, the sequence $\{u^k\}$ converges.
\end{lemma}
\begin{lemma}[see \cite{Polyak1987}]\label{lemmasec}
	Let $\{\lambda_k\}$ be a sequence of positive numbers, $\{a_k\}$ a sequence of real numbers, and $b_n:=\sigma^{-1}_n\sum_{k=1}^n\lambda_ka_k$, 
	where $\sigma_n:=\sum_{k=1}^n\lambda_k$. If $\sigma_n\to\infty$, $\liminf a_n\leq\liminf b_n\leq\limsup b_n\leq\limsup a_n$.
\end{lemma}

\section{On the algorithms and convergence analysis}\label{sec12}

Let $C\subset\mathbb{R}^n$ be an open nonempty convex set. From now on,  $f:\mathbb{R}^n\to\mathbb{R}$ is a lower semicontinuous bounded below $DC$ function and $g, h : \mathbb{R}^n\to\mathbb{R}$ are lower semicontinuous and convex functions such that $f(x)=g(x)-h(x)$. In addition, in all further results, assume that $(d,H)\in\Phi_{+}(C)$. 
 
 To solve the problem of finding a critical point of $f$ on $\bar C$, we will study the following algorithms:

\textbf{Algorithm 1:}\\
 Let $\lambda_k>0$, $k\in\mathbb{N}$. Start from a point $x^0\in C$ and generates a sequence $\{x^k\}\subset C$ satisfying
$$
v^k\in \partial g(x^{k}),\quad w^k\in\partial h(x^k),
$$
$$
x^{k+1}\in\mbox{argmin}\left\{\lambda_k\langle v^k-w^k,z\rangle+d(z,x^k)~|~z\in C\right\}.
$$ 

\textbf{Algorithm 2:}\\
 Let $\lambda_k>0$, $k\in\mathbb{N}$. Take a inicial point $x^0\in C$ and generates a sequence $\{x^k\}\subset C$ satisfying
$$
\quad w^k\in\partial h(x^k),
$$
$$
x^{k+1}\in\mbox{argmin}\left\{g(x)-\langle w^k,z-x^k\rangle+(1/\lambda_k)d(z,x^k)~|~z\in C\right\}.
$$
\begin{remark}
The existence of $\{x^k\}\subset C$ of both Algorithm 1 and 2 is guaranteed by using similar arguments as in the proof of \cite[Proposition 2.1]{Auslander2006} and \cite[Proposition 2.3]{CruzNeto2018}. Then, from optimality conditions, we obtain
\begin{equation}\label{france11}
\lambda_k(v^k-w^k)+\nabla_1d(x^{k+1},x^k)=0,\quad k\in\mathbb{N},
\end{equation} 
for Algorithm 1, and for Algorithm 2, there exists $z^{k+1}\in\partial g(x^{k+1})$, such that
\begin{equation}\label{france22}
\lambda_k(z^{k+1}-w^k)+\nabla_1d(x^{k+1},x^k)=0,\quad k\in\mathbb{N}.
\end{equation} 
\end{remark} 

Next we present an important result to our convergence analysis.
\begin{proposition}\label{prop1} Set 
$\beta_k:=\langle \nabla_1d(x^{k+1},x^k),x^{k+1}-x^{k}\rangle$ and assume that $h$ is strongly convex with modulus $\rho$. Then the following hold:
\begin{description}
\item [\rm(i)] For all $k\in\mathbb{N}$, 
$\beta_k\geq0$. Furthermore, assume that there exists a positive constant $\kappa$ satisfying
\begin{equation}\label{klipschitz}
\partial g(x)\subset\partial g(y)+\|x-y\|\mathbb{B}\quad\forall x,y\in C,
\end{equation}
where $\mathbb{B}$ denotes the closed unit ball in $\mathbb{R}^n$. Then,
\begin{equation}\label{france}
f(x^k)-f(x^{k+1})\geq(\rho-\kappa)\left\|x^k-x^{k+1}\right\|^2+\frac{\beta_k}{\lambda_k},\quad k\in\mathbb{N},
\end{equation}
for Algorithm 1, and 
\begin{equation}\label{france1}
f(x^k)-f(x^{k+1})\geq\rho\left\|x^k-x^{k+1}\right\|^2+\frac{\beta_k}{\lambda_k},\quad k\in\mathbb{N},
\end{equation}
for Algorithm 2.
\item [\rm(ii)] Assume that $\lambda_k\leq\lambda^+$, $k\in\mathbb{N}$. Then, $\sum_k\|x^k-x^{k+1}\|^2<\infty$, $\sum_k\beta_k<\infty$, for Algorithm 2. Besides, if $\rho>\kappa$ , $\sum_k\beta_k<\infty$, for Algorithm 1.
\end{description}

\end{proposition}
\begin{proof} Let us prove (i). From \textbf{(H2)},  with $z=x=x^k$, $ y=x^{k+1}$, and taking into account that $H(x^k,x^k)=0$, we obtain
	\begin{equation}\label{fadas23}
	H(x^k,x^{k+1})\leq\langle x^{k+1}-x^k,\nabla_1 d(x^{k+1},x^k)\rangle,\quad k\in\mathbb{N}.
	\end{equation}
	Since $H(x^k,x^{k+1})\geq0$, we have that $\beta_k\geq0$, $k\in\mathbb{N}$. Now, let us prove \eqref{france}. First, in view of \eqref{france11}, we have
\begin{equation}\label{paris}
\lambda_k(v^k-w^k)+\nabla_1 d(x^{k+1},x^k)=0,\quad k\in\mathbb{N}.
\end{equation}
Since $x^k\in C$ for all $k\geq0$, we can use \eqref{klipschitz}, to obtain
$$
\partial g(x^k)\subset\partial g(x^{k+1})+\kappa\left\|x^k-x^{k+1}\right\|\mathbb{B},\quad k\in\mathbb{N}.
$$
Taking into account that $v^k\in \partial g(x^{k})$, last inclusion implies that there exist $u^k\in\partial g(x^{k+1})$ and $b^k\in\mathbb{B}$ satisfying
\begin{equation}\label{london}
v^k=u^k+\kappa\left\|x^k-x^{k+1}\right\|b^k,\quad k\in\mathbb{N}.
\end{equation}
From convexity of $g$,
$$
g(x^k)\geq g(x^{k+1})+\langle u^k,x^k-x^{k+1}\rangle,\quad k\in\mathbb{N}.
$$
Now, combining last inequality with \eqref{london}, we obtain
$$
g(x^k)\geq g(x^{k+1})+\langle v^k,x^k-x^{k+1}\rangle-\kappa\left\|x^k-x^{k+1}\right\|\langle b^k,x^k-x^{k+1}\rangle,\quad k\in\mathbb{N}.
$$
Consequently, from \eqref{paris}, we have
\begin{eqnarray*}
g(x^k)&\geq& g(x^{k+1})-\frac{1}{\lambda_k}\langle \nabla_1 d(x^{k+1},x^k),x^k-x^{k+1}\rangle\\&+&\langle w^k,x^k-x^{k+1}\rangle-\kappa\left\|x^k-x^{k+1}\right\|\langle b^k,x^k-x^{k+1}\rangle,\quad k\in\mathbb{N}.
\end{eqnarray*}
Since $\beta_k=\langle \nabla_1d(x^{k+1},x^k),x^{k+1}-x^{k}\rangle$, we have
\begin{eqnarray*}
g(x^k)&\geq& g(x^{k+1})+\langle w^k,x^k-x^{k+1}\rangle\\
&-&\kappa\left\|x^k-x^{k+1}\right\|\langle b^k,x^k-x^{k+1}\rangle+\frac{\beta_k}{\lambda_k},\quad k\in\mathbb{N}.
\end{eqnarray*}
On the other hand, as $h$ is strongly convex with modulus $\rho>0$, we have
\begin{equation}\label{hineq}
h(x^{k+1})\geq h(x^k)+\langle w^k,x^{k+1}-x^k\rangle+\rho\left\|x^k-x^{k+1}\right\|^2,\quad k\in\mathbb{N}.
\end{equation}
Then, we obtain
\begin{eqnarray*}
g(x^k)&\geq& g(x^{k+1})+h(x^k)-h(x^{k+1})-\kappa\left\|x^k-x^{k+1}\right\|\langle b^k,x^k-x^{k+1}\rangle\\
&+&\rho\|x^k-x^{k+1}\|^2+\frac{\beta_k}{\lambda_k},\quad k\in\mathbb{N}.
\end{eqnarray*}
Using Cauchy--Schwartz inequality, we have
$$
g(x^k)-h(x^k)\geq g(x^{k+1})-h(x^{k+1})+(\rho-\kappa)\left\|x^k-x^{k+1}\right\|^2+\frac{\beta_k}{\lambda_k},\quad k\in\mathbb{N}.
$$
Finally, since $f(x)=g(x)-h(x)$ we obtain \eqref{france}. 

To prove \eqref{france1}, we can use the convexity of $g$ combined with \eqref{france22} to obtain,
$$
g(x^k)\geq g(x^{k+1})+\langle w^k,x^k-x^{k+1}\rangle
-\frac{1}{\lambda_k}\langle\nabla_1 d(x^{k+1},x^k),x^k-x^{k+1}\rangle,\quad k\in\mathbb{N}.
$$
Then, use last inequality with \eqref{hineq} to obtain \eqref{france1}.

The item (ii), follows immediately from (i), and from the fact that $f$ is bounded below.

\end{proof}

\begin{theorem}\label{theo1}
Under all the assumptions of Proposition \ref{prop1}, suppose furthermore that $\lambda_k\geq\lambda_->0$, $k\in\mathbb{N}$. If $\{x^k\}$ is generated by Algorithm 1 or Algorithm 2, its accumulation points, if any, are critical points of $f$ in $\bar C$.
\end{theorem}
\begin{proof} Let $\bar x$ be an accumulation point of $\{x^k\}$ and let $\{x^{k_j}\}$ a subsequence of $\{x^k\}$ such that $\lim_{j\to\infty}x^{k_j}=\bar x$. From Algorithms 1 and 2, $v^{k_j}\in \partial g(x^{k_j})$, $ w^{k_j}\in\partial h(x^{k_j})$ and $z^{k_j+1}\in\partial g(x^{k_j+1})$. Now, thanks do Proposition \ref{prop1} (ii), $\lim_{j\to\infty}x^{k_j+1}=\bar x$. Then, we can use \cite[Theorem 9.13]{RockafellarWets1998} and, without loss of generality, we can assume that $\{v^{k_j}\}$, $\{w^{k_j}\}$ and $\{z^{k_j+1}\}$ converge to $\bar v$, $\bar w$ and $\bar z$, respectively.

Now, let us prove the result for Algorithm 1. Consider any $y_0\in \bar{C}$ fixed. Based on \eqref{france11},
	$$
	\langle v^k-w^k,y_0-x^k\rangle=-\frac{1}{\lambda_k}\langle \nabla_1d(x^{k+1},x^k),y_0-x^{k}\rangle,\quad k\in\mathbb{N}.
	$$
	Then, tanking into account that $\beta_k:=\langle \nabla_1d(x^{k+1},x^k),x^{k+1}-x^{k}\rangle$, we have
	\begin{equation}\label{eq:method3453}
	\langle v^k-w^k,y_0-x^k\rangle=-\frac{1}{\lambda_k}\langle \nabla_1d(x^{k+1},x^k),y_0-x^{k+1}\rangle-\frac{\beta_k}{\lambda_k},\quad k\in\mathbb{N}.
	\end{equation}
	From \textbf{(H2)},  with $z=y_0, y=x^{k+1}, x=x^k$, we obtain
	\begin{equation}\label{fadas}
	\langle y_0-x^{k+1},\nabla_1 d(x^{k+1},x^k)\rangle\leq H(y_0,x^k)-H(y_0,x^{k+1}),\quad k\in\mathbb{N}.
	\end{equation} 
	Combining the last inequality with  \eqref{eq:method3453}, for all $k\geq 0$ we obtain
	\begin{eqnarray*}\label{solista}
	H(y_0,x^{k})-H(y_0,x^{k+1})+\beta_k\geq\nonumber
	-\lambda_k \langle v^k-w^k,y_0-x^k\rangle.
	\end{eqnarray*} 
	Summing the last inequality over $k=1,\dots,n$, for all $k\geq0$ we have
	\begin{eqnarray*}
	H(y_0,x^{1})-H(y_0,x^{n+1})+\sum_{k=1}^{n}\beta_k\geq\nonumber
	\sum_{k=1}^{n}\lambda_k\left(-\langle v^k-w^k,y_0-x^k\rangle\right).
	\end{eqnarray*}
	Since $H(\cdot,\cdot)\geq0$, for all $k\geq0$ we obtain 
	\begin{eqnarray*}
	\sigma_n^{-1}H(y_0,x^{1})+\sigma_n^{-1}\sum_{k=1}^{n}\beta_k\geq\nonumber	\sigma_n^{-1}\sum_{k=1}^{n}\lambda_k\left(-\langle v^k-w^k,y_0-x^k\rangle\right),
	\end{eqnarray*}
	where $\sigma_n:=\sum_{k=1}^{n}\lambda_k$. As $\lambda_k\geq\lambda_-$, then $\sigma_n\to\infty$, and considering that $\sum_{k=1}^{\infty}\epsilon_k<\infty$, we can use  Lemma \ref{lemmasec} to obtain 
	$$
	\limsup_{k\to+\infty}\langle v^k-w^k,y_0-x^k\rangle\geq0.
	$$
	Since $\partial g$ and $\partial h$ are closed, we obtain $\bar v\in\partial g(\bar x)$ and  $\bar w\in\partial h(\bar x)$. Thus, last inequality imples that $$\langle \bar v-\bar w,y_0-\bar x\rangle\geq0,$$ for all $y_0\in \bar C$. Therefore $\bar x\in \mathcal{S}_{\bar C}^*(f)$. 

Now, for Algorithm 2, again, consider any $y_0\in \bar{C}$ fixed. Based on \eqref{france22},
	$$
	\langle z^{k+1}-w^k,y_0-x^{k+1}\rangle=-\frac{1}{\lambda_k}\langle \nabla_1d(x^{k+1},x^k),y_0-x^{k+1}\rangle,\quad k\in\mathbb{N}.
$$
Using the same arguments as in the Algorithm 1, we obtain  
	\begin{eqnarray*}
	\sigma_n^{-1}H(y_0,x^{1})\geq\nonumber	\sigma_n^{-1}\sum_{k=1}^{n}\lambda_k\left(-\langle z^{k+1}-w^k,y_0-x^k\rangle\right),
	\end{eqnarray*}
	where $\sigma_n:=\sum_{k=1}^{n}\lambda_k$. The rest of the proof is exactly the same as was done for Algorithm 1. 
\end{proof}

\begin{lemma}\label{lemma3} Under all the assumptions of Proposition \ref{prop1}, suppose furthermore that $g$ is strongly convex with modulus $\gamma>0$, $h$ is continuously differentiable and $\nabla h$ is $L$--Lipschitz continuous on $C$. Consider any $\bar x\in \mathcal{S}_{\bar C}^*(f)$.
Then the following hold:
\begin{description}
\item [\rm(i)] For Algorithm 1, 
\begin{equation}\label{france81}
H(\bar x,x^{k+1})+\lambda_k(\gamma-L)\|x^k-\bar x\|^2\leq H(\bar x,x^k)+\beta_k,\quad k\in\mathbb{N}.
\end{equation}
\item [\rm(ii)] For Algorithm 2,
\begin{equation}\label{france82}
H(\bar x,x^{k+1})+\lambda_k(\gamma-L-1/2)\|x^{k+1}-\bar x\|^2\leq H(\bar x,x^k)+\alpha_k,\quad k\in\mathbb{N},
\end{equation}
with $\alpha_k:=\lambda_k/2\|x^k-x^{k+1}\|^2$.
\end{description}
\end{lemma}
\begin{proof}
Take any $\bar x\in \mathcal{S}_{\bar C}^*(f)$ and let $v\in\partial g(\bar x)$ be such that, for all $y\in C$, 
$$
\langle v-\nabla h(\bar x),y-\bar x\rangle\geq0.
$$
Let us prove (i). Since $\{x^k\}\subset C$, we obtain
$\langle v,x^k-\bar x\rangle\geq\langle\nabla h(\bar x),x^k-\bar x\rangle$, $k\in\mathbb{N}$. 
Since $g$ is strongly convex with modulus $\gamma$, we have $\gamma\|x^k-\bar x\|^2\leq\langle x^k-\bar x, v^k-v\rangle$, $k\in\mathbb{N}$. Consequently, $\gamma\|x^k-\bar x\|^2\leq\langle x^k-\bar x, v^k-\nabla h(\bar x)\rangle$, $k\in\mathbb{N}$.
From \eqref{france11},
$$
\gamma\|x^k-\bar x\|^2\leq-\frac{1}{\lambda_k}\langle x^k-\bar x,\nabla_1 d(x^{k+1},x^k)\rangle+\langle x^k-\bar x,\nabla h(x^k)-\nabla h(x)\rangle,\quad k\in\mathbb{N}.
$$
Taking into account that $\beta_k=\langle \nabla_1d(x^{k+1},x^k),x^{k+1}-x^{k}\rangle$, for all $k\in\mathbb{N}$,
$$
\gamma\|x^k-\bar x\|^2\leq \frac{\beta_k}{\lambda_k}-\frac{1}{\lambda_k}\langle x^{k+1}-\bar x,\nabla_1 d(x^{k+1},x^k)\rangle+\langle x^k-\bar x,\nabla h(x^k)-\nabla h(\bar x)\rangle.
$$
As $\nabla h$ is $L$--Lipschitz continuous on $C$ and using Cauchy-Schwarz, we obtain
$$
\gamma\|x^k-\bar x\|^2\leq \frac{\beta_k}{\lambda_k}+\frac{1}{\lambda_k}\langle \bar x-x^{k+1},\nabla_1 d(x^{k+1},x^k)\rangle+L\|x^k-\bar x\|^2,\quad k\in\mathbb{N}.
$$
Again using \textbf{(H2)},  with $z=\bar x, y=x^{k+1}, x=x^k$, we have
\begin{equation}\label{rosa}
\langle \bar x-x^{k+1},\nabla_1 d(x^{k+1},x^k)\rangle\leq H(\bar x,x^k)-H(\bar x,x^{k+1}),\quad k\in\mathbb{N}.
\end{equation}
	Hence, we can combine the last two inequalities to obtain \eqref{france81}.

Now let us prove (ii).
Since $\{x^k\}\subset C$, 
$\langle v,x^{k+1}-\bar x\rangle\geq\langle\nabla h(\bar x),x^{k+1}-\bar x\rangle$, for all $k\in\mathbb{N}$. Now, 
$$\gamma\|x^{k+1}-\bar x\|^2\leq\langle x^{k+1}-\bar x, z^{k+1}-v\rangle,\quad k \in\mathbb{N},$$ thanks to the strongly convexity of $g$. Thus,
$\gamma\|x^{k+1}-\bar x\|^2\leq\langle x^{k+1}-\bar x, z^{k+1}-\nabla h(\bar x)\rangle$, $k\in\mathbb{N}$. Hence, taking into account that \eqref{france22} holds, for all $k\in\mathbb{N}$,
\begin{equation}\label{leque}
\gamma\|x^{k+1}-\bar x\|^2\leq-\frac{1}{\lambda_k}\langle x^{k+1}-\bar x,\nabla_1 d(x^{k+1},x^k)\rangle+\langle x^{k+1}-\bar x,\nabla h(x^k)-\nabla h(\bar x)\rangle.
\end{equation}
On the other hand, for all $k\in\mathbb{N}$,
\begin{eqnarray*}
\langle x^{k+1}-\bar x,\nabla h(x^k)-\nabla h(\bar x)\rangle&=&\langle x^{k+1}-\bar x,\nabla h(x^{k+1})-\nabla h(\bar x)\rangle\\&+&
\langle x^{k+1}-\bar x,\nabla h(x^k)-\nabla h(x^{k+1})\rangle.
\end{eqnarray*}
As $\nabla h$ is $L$--Lipschitz continuous on $C$ and using Cauchy-Schwarz inequality, we obtain
\begin{eqnarray*}
\langle x^{k+1}-\bar x,\nabla h(x^k)-\nabla h(\bar x)\rangle
&\leq& L\|x^{k+1}-\bar x\|^2+\|x^{k+1}-\bar x\|\|x^k-x^{k+1}\|.
\end{eqnarray*}
Now, since $ab\leq1/2(a^2+b^2)$, $a,b\geq0$,
\begin{eqnarray*}
\langle x^{k+1}-\bar x,\nabla h(x^k)-\nabla h(\bar x)\rangle\leq (L+1/2)\|x^{k+1}-\bar x\|^2+1/2\|x^k-x^{k+1}\|^2.
\end{eqnarray*}
Combining last inequality with \eqref{leque}, we obtain
$$
(\gamma-L-1/2)\|x^{k+1}-\bar x\|^2\leq\frac{1}{\lambda_k}\langle \bar x-x^{k+1},\nabla_1 d(x^{k+1},x^k)\rangle+1/2\|x^k-x^{k+1}\|^2
$$
Thus, \eqref{france82} can be obtained combining last inequality with \eqref{rosa}.
\end{proof}

To set the  convergence of  any sequence  generated by Algorithm 1 and  2,  we need to make further assumptions on the induced proximal
distance $H$, which  were also considered in \cite{Auslander2006}. Let $(d,H)\in \Phi_{+}(\bar{C})\subset\Phi(\bar{C})$ be such that the function $H$ satisfies the following
two additional properties:  For  $ y\in \bar{C}$ and $ \{y^k\}\subset C$, 
\begin{description}
	\item [(\textbf{Ha})]  $\lim_{k\to+\infty}y^k=y$, whenever   $ \{y^k\}$  is bounded  and  $\lim_{k\to+\infty}H(y,y^k)=0$;
	\item [(\textbf{Hb})]  $\lim_{k\to+\infty}H(y,y^k)=~0$, whenever $\lim_{k\to+\infty}y^k=y$.
\end{description}

We also make the following assumption:
\begin{equation}\label{lasth}
\mathcal{S}_{\bar C}^*(f)\neq\emptyset.
\end{equation}
	
Under  these  assumptions,  we prove that both Algorithm 1 and 2 converges to a Clarke critical of $f$.

\begin{theorem}
Under all the assumptions of Lemma \ref{lemma3}, suppose furthermore that $0<\lambda_-\leq\lambda_k\leq\lambda^+$, $k\in\mathbb{N}$ and $\gamma>L+1/2$. If $\{x^k\}$ is generated by Algorithm 1 or Algorithm 2, then it converges to a Clarke critical point point of $f$ in $\bar C$.
\end{theorem}
\begin{proof} In view of \eqref{lasth}, take any $x\in\mathcal{S}_{\bar C}^*(f)$. As $\gamma>L+1/2$, Lemma \ref{lemma3} implies that,
$H(x,x^{k+1})\leq H(x,x^k)+\beta_k$, and $H(x,x^{k+1})\leq H(x,x^k)+\alpha_k$,  for all $k\in\mathbb{N}$. As $\sum_k\alpha_k<\infty$, and thanks to Proposition \ref{prop1} (ii), in both cases, we can apply Lemma \ref{lemmaone} we conclude  that $\{H(x,x^{k})\}$ converges to some point $\beta(x)$. Let $x^*$ be an accumulation point of $\{x^{k}\}$. From Theorem \ref{theo1}, 
$x^*\in\mathcal{S}_{\bar C}^*(f)$. Based on \textbf{(Ha)}, we obtain 
	$\lim_{\ell\to+\infty}H(x^*,x^{k_{\ell}})=0$.  Considering that  $\{H(x,x^{k})\}$ converges, we conclude that  $\lim_{k\to+\infty}H(x^*,x^{k})=0$. Now,  by \textbf{(Hb)} it follows that $\{x^k\}$ converges to $x^*$.  Therefore, from Theorem~\ref{theo1}, $x^*$ is a Clarke critical point of $f$ in $\bar{C}$, which  proves the theorem.	
	
\end{proof}

\section{Conclusions} \label{sec:conclusion}

In this paper, we present an interior subgradient and a proximal linearized method for $DC$ programming, whose regularization term is a proximal distance. Based on the methods presented in  \cite{Auslander2006,CruzNeto2018}, we prove that any accumulation point of the respective sequences of both methods is a critical point in the sense of Definition \ref{defi1}, where the strong convexity of one of the components of the main function played a vital role in this analysis. It is worth to point out that, for Algorithm 1, it was supposed that, the subdifferencial of one of the component functions, had the locally Lipschitz property on the constrained set. This assumption, in the differentiable setting, is commonly used in gradient-type algorithms. Finally, in the differentiable setting, we prove that the whole sequence of both methods converges to a Clarke-critical point. In future research, we intend to investigate this kind of problem in more general settings as in Riemann Manifolds and Multi-objective Optimization.  We foresee further progress in this topic in the near future.



\end{document}